\documentclass[a4paper,10pt,reqno]{article}

\usepackage[utf8]{inputenc}
\usepackage[T1]{fontenc}
\usepackage{lmodern}
\usepackage[english]{babel}
\usepackage{microtype}

\usepackage{amsmath,amssymb,amsfonts,amsthm}
\usepackage{mathtools,accents}
\usepackage{mathrsfs}
\usepackage{aliascnt}
\usepackage{braket}
\usepackage{bm}
\usepackage[normalem]{ulem}
\usepackage{authblk}
\usepackage{cancel}
\usepackage[a4paper,margin=3cm]{geometry}
\usepackage[citecolor=blue,colorlinks]{hyperref}

\usepackage{enumerate}
\usepackage{xcolor}

\makeatletter
\g@addto@macro\@floatboxreset\centering
\makeatother

\makeatletter
\def\newaliasedtheorem#1[#2]#3{
  \newaliascnt{#1@alt}{#2}
  \newtheorem{#1}[#1@alt]{#3}
  \expandafter\newcommand\csname #1@altname\endcsname{#3}
}
\makeatother

\numberwithin{equation}{section}

\newtheoremstyle{slanted}{\topsep}{\topsep}{\slshape}{}{\bfseries}{.}{.5em}{}

\theoremstyle{plain}
\newtheorem{theorem}{Theorem}[section]
\newaliasedtheorem{proposition}[theorem]{Proposition}
\newaliasedtheorem{lemma}[theorem]{Lemma}
\newaliasedtheorem{corollary}[theorem]{Corollary}
\newaliasedtheorem{counterexample}[theorem]{Counterexample}

\theoremstyle{definition}
\newaliasedtheorem{definition}[theorem]{Definition}
\newaliasedtheorem{question}[theorem]{Question}
\newaliasedtheorem{conjecture}[theorem]{Conjecture}

\theoremstyle{remark}
\newaliasedtheorem{remark}[theorem]{Remark}
\newaliasedtheorem{example}[theorem]{Example}
\newaliasedtheorem{assumption}[theorem]{Assumption}

\newcommand{\setN}{\mathbb{N}}

\newcommand{\R}{\mathbb{R}}

\newcommand{\eps}{\varepsilon}

\let\altphi\phi
\let\phi\varphi
\let\varphi\altphi
\let\altphi\undefined

\newcommand{\di}{\mathop{}\!\mathrm{d}}

\DeclareMathOperator{\supp}{supp}

\newcommand{\Ch}{{\sf Ch}}

\newcommand{\Opt}{\mathrm{OptGeo}}

\newcommand{\meass}{(X,\sfd,\meas)}
\newcommand{\model}{(I_{K,N},\sfd_{eu},\meas_{K,N})}

\DeclareMathOperator{\Geo}{Geo}
\DeclareMathOperator{\Lip}{Lip}

\DeclareMathOperator{\Per}{Per}

\newcommand{\Prob}{\mathscr{P}}
\newcommand{\Borel}{\mathscr{B}}

\newcommand{\dist}{\mathsf{d}}

\DeclareMathOperator{\sfd}{\dist}
\newcommand{\meas}{\mathfrak{m}}

\DeclareMathOperator{\CD}{CD}
\DeclareMathOperator{\RCD}{RCD}
\newfont{\tmpf}{cmsy10 scaled 2500}

\DeclareMathOperator*{\esssup}{ess\,sup}

\begin{document}

\title{\large \bf A reverse Hölder inequality for first eigenfunctions of the Dirichlet Laplacian on $\RCD(K,N)$ spaces}

\date{ }

\author{
\small  MUSTAFA ALPER GUNES\thanks{ Mathematical Institute, University of Oxford, UK. email: mustafa.gunes@st-hildas.ox.ac.uk} 
 \quad and \quad   \small ANDREA MONDINO \thanks{Mathematical Institute, University of Oxford, UK.  email: Andrea.Mondino@maths.ox.ac.uk}
}

\maketitle

\begin{abstract}

In the framework of  (possibly non-smooth) metric measure spaces with Ricci curvature bounded below by a positive constant in a synthetic sense, we establish a sharp and rigid reverse-H\"older inequality for first eigenfunctions of the Dirichlet Laplacian. This generalises to the positively curved and non-smooth setting the classical ``Chiti Comparison Theorem''. We also prove a related quantitative stability result  which seems to be new even for smooth Riemannian manifolds.
\end{abstract}

\section{Introduction}

The goal of the paper is to establish a reverse H\"older inequality for first eigenfunctions of the Dirichlet Laplacian in the framework on (possibly non-smooth) spaces satisfying a positive Ricci curvature lower bound in a synthetic sense. More precisely, the setting of the paper is given by the so-called $\mathrm{RCD}(K,N)$ metric measure spaces: the real parameter  $N\in(1,\infty)$ plays the role of (synthetic) upper bound  on the dimension, and the real parameter $K\in {\mathbb R}$ plays the role of (synthetic) lower bound on the Ricci curvature. We refer the reader to Section \ref{sec:prel} for a precise definition. For the sake of the introduction, let us just recall that a metric measure space is a triplet $(X,\sfd,\meas)$ where $(X,\sfd)$ is a complete separable metric space and $\meas$ is a Borel non-negative measure; for simplicity, we will assume that $\supp(\meas)=X$ and $\meas(X)=1$ throughout. Let us also mention that the class of $\mathrm{RCD}(K,N)$ metric measure spaces includes, as remarkable examples:
\begin{itemize}
\item (possibly weighted) Riemannian manifolds with (Bakry-\'Emery) Ricci curvature bounded below and their pointed-measured-Gromov-Hausdorff limits;
\item finite dimensional Alexandrov spaces with curvature bounded below.
\end{itemize}
When compared with the class of smooth Riemannian manifolds with Ricci curvature bounded from below, a remarkable advantage (that will also be used to establish some of the results of this note) of the class of $\mathrm{RCD}(K,N)$ metric measure spaces is that it is closed under natural operations which typically create singularities, such as taking pointed-measured-Gromov-Hausdorff limits, or taking quotients under groups of isometries.

On an $\mathrm{RCD}(K,N)$ metric measure space there is a naturally defined Laplacian operator (obtained via integration by parts of the so-called Cheeger energy, see section Section \ref{sec:prel} for more details).
\\ Given an open subset $\Omega\subset X$ of an $\mathrm{RCD}(K,N)$ space   $(X,\sfd,\meas)$, for some $K>0, \, N\in (1,\infty)$, we will consider solutions to the eigenvalue problem:

\begin{equation}
\begin{cases}
    -\Delta u=\lambda u \quad  \text{on }   \Omega \\
    \quad  \; \;u=0  \quad  \; \; \,  \text{on } \partial \Omega\, .
    \label{Laplacian}
    \end{cases}
\end{equation}

Our aim will be to draw comparison results between a solution $u$ of \eqref{Laplacian}, with a solution to a corresponding eigenvalue problem on the one-dimensional model metric measure space, $(I,\sfd_{eu},\meas_{K,N})$, with parameters $K>0$ and $N \in (1,\infty)$ given by:
\begin{equation}\label{eq:defIKN}
    \begin{cases}
    I_{K,N}= \left[0,\sqrt{\frac{N-1}{K}}\pi\right] \subset \R \\
   \meas_{K,N}=\frac{1}{c_{K,N}} \sin^{N-1}\left(t\sqrt{\frac{K}{N-1}}\right) \mathcal{L}^1(\di t) \\
    \sfd_{eu}(x,y)=|x-y|\, ,
    \end{cases}
\end{equation}
where $\mathcal{L}^1$ denotes the standard Lebesgue measure on $I_{K,N}$ and $c_{K,N}$ is a normalizing constant so that $\meas_{K,N}$ is a probability measure on $I_{K,N}$.
Notice that, in case $N\geq 2$ is integer, the model space $(I_{K,N}, \sfd_{eu}, \meas_{K,N})$ corresponds in a natural way to the $N$-dimensional round sphere of radius $\sqrt{\frac{N-1} {K}}$.
\\
We will be interested in the solution of the problem \eqref{Laplacian} for an open set $\Omega \subset X$ when $\lambda$ is the smallest positive real number such that a non-trivial solution exists, i.e. when $\lambda$ is the first Dirichlet eigenvalue, that we will denote by  $\lambda_X(\Omega)$. The following variational characterization of $\lambda_X(\Omega)$ and of the corresponding quantity in the model space is very useful:
\begin{definition}\label{def:VarCharlambda}
Let $(X,\sfd,\meas)$ be an $\mathrm{RCD}(K,N)$ space with $K>0, N \in (1,\infty)$ and let $\Omega \subset X$ be open. The first Dirichlet eigenavalue on $\Omega$ is given by
\begin{equation}\label{eq:1stEVRCD}
    \lambda_X(\Omega)=\inf_u \left\{\frac{\int_\Omega |\nabla u|^2 \di \meas}{\int_\Omega u^2 \di \meas}\right\}\, ,
\end{equation}
where the infimum is taken over all (non-identically vanishing) functions $u: \Omega\to \R$ that are Lipschitz with compact support on $\Omega$. Similarly, for $v\in (0,1)$, letting $r(v)$ be such that $\meas_{K,N}[0,r(v)]=v$, the first Dirichlet eigenvalue on the model space $([0,r(v)],\sfd_{eu},\meas_{K,N})$ is given by:
\begin{equation}
   \lambda_{K,N,v}=\inf_u \left\{\frac{\int_0^{r(v)} |u'|^2 \, \di\meas_{K,N}}{\int_0^{r(v) }u^2\,  \di\meas_{K,N}}\right\}\, ,
\end{equation}
where the infimum is taken over all  (non-identically vanishing) Lipschitz functions $u:[0,r(v)]\to [0,\infty)$ such that $u(r(v))=0$.
\end{definition}

\subsection{Main Theorems}
Let $(X,\sfd,\meas)$ be an $\mathrm{RCD}(K,N)$ space for some $K>0,\, N\in (1,\infty)$  and $\Omega\subset X$ be an open set such that $\meas(\Omega)=v$. It was proved in \cite[Theorem 1.5]{Polya} that:
\begin{equation*}
  \lambda_{K,N,v}\leq \lambda_X(\Omega).
\end{equation*}
From the variational characterization of the first eigenvalue \eqref{eq:1stEVRCD},  the fact that $\lambda_{K,N,t} \to \infty$ as $t\to 0$, and that $t \mapsto \lambda_{K,N,t}$ is a continuous function, it follows that there exists $\alpha \in (0,v]$ such that  $\lambda_{K,N,\alpha}=\lambda_X(\Omega)$.  Furthermore (see for instance \cite[Theorem 5.6]{Polya}) the eigenfunction on $([0,r(\alpha)],\sfd_{eu},\meas_{K,N})$ corresponding to the first eigenvalue $\lambda_{K,N,\alpha}$, which we will denote by $z(x)$ throughout the rest of the paper, is unique up to multiplication by a constant.\\
For simplicity of notation, the results will be stated in the case $K=N-1$ and $\meas(X)=1$.  This does not compromises generality, indeed every $\mathrm{RCD}(K,N)$ for some $K>0$ is compact with finite measure and can be scaled to become an $\mathrm{RCD}(N-1,N)$ space endowed with a probability measure. More precisely,  it holds  that $(X,\sfd,\meas)$ is an $\mathrm{RCD}(K,N)$ space with $K>0$, $N\in (1,\infty)$ if and only if, letting $\Tilde{\sfd}=\sqrt{\frac{N-1}{K}}\sfd$ and $\Tilde{\meas}= \meas(X)^{-1}\, \meas$,  the scaled and normalised metric measure space $(X,\Tilde{\sfd},\tilde{\meas})$ satisfies the $\mathrm{RCD}(N-1,N)$ condition.
\\ We are finally in position to state the first main result of the note. 

\begin{theorem}[A reverse H\"older inequality for first Dirichlet eingenfunctions]
Let $(X,\sfd,\meas)$ be an $\mathrm{RCD}(N-1,N)$ space, let $\Omega \subset X$ be an open set such that $\meas(\Omega)=v \in (0,1)$. Let $\alpha\leq v$ be such that $\lambda_{N-1,N,\alpha}=\lambda_X(\Omega)$. Let $u$ be a first Dirichlet eigenfunction on $\Omega$. Let $p>0$ and let $z:[0,r(\alpha)]\to \R$ be the eigenfunction corresponding to $\lambda_{N-1,N,\alpha}$, scaled so that:
\begin{equation}
  \int_{\Omega}u^p \di \meas=\int_0^{r(\alpha)} z^p \di \meas_{N-1,N}\, .
  \label{eq:hölderintegrals}  
\end{equation}
Then, for all $q\geq p$, it holds that:
\begin{equation}
    \frac{\left(\int_{\Omega}u^q \di \meas\right)^{\frac{1}{q}}}{\left(\int_{\Omega}u^p \di \meas\right)^{\frac{1}{p}}} \leq \frac{\left(\int_{0}^{r(\alpha)}z^q \di \meas_{N-1,N}\right)^{\frac{1}{q}}}{\left(\int_{0}^{r(\alpha)}z^p \di \meas_{N-1,N}\right)^{\frac{1}{p}}} \,. 
    \label{eq:holderineq}
\end{equation}
Further, if the equality holds for some $q>p$ then $(X,\sfd,\meas)$ is isomorphic to a spherical suspension.
\label{mainthm}
\end{theorem}

The second main result is a  stable version of Theorem \ref{mainthm}, when $p=1$.
\begin{theorem}[Quantitative stability in the reverse Hölder inequality]\label{thm:ChitiStability} Let $N\in (1,\infty)$, $v\in (0,1)$ and $\lambda >0$ be given. Then there exists $C=C(N,v,\lambda)>0$ and $\tilde{\delta}=\tilde{\delta}(N,v,\lambda)>0$ such that the following holds.

Let $(X,\sfd,\meas)$ be an $\mathrm{RCD}(N-1,N)$ space. Let $\Omega \subset X$ be an open domain with $\meas\left(\Omega\right)=v\in (0,1)$ and $\lambda_{X}(\Omega)=\lambda$.  Let $u$ and $z$ be as in the assumptions of Theorem \ref{mainthm}, scaled so that:
\begin{equation*}
    \int_\Omega u \di \meas=\int_0^{r(\alpha)}z \di \meas_{N-1,N}=1\, .
\end{equation*}
If there exist $\delta\in (0, \tilde{\delta})$ and an unbounded subset $Q\subset (0,\infty)$ such that 
\begin{equation}  \label{eq:stabilityineq}
  \|z\|_{L^q\left((0,r(\alpha)),\meas_{N-1,N}\right)}-\|u\|_{L^q\left(\Omega,\meas\right)}<\delta \,, \quad \text{for all $q\in Q$ , }
\end{equation}
then
\begin{equation*}
    \left(\pi-{\rm{diam}}(X)\right)^N \leq C\sqrt{\delta}\, .
\end{equation*}
In particular, for all $\eps>0$ there exists $\delta=\delta(N,v,\lambda,\eps)>0$ such that if \eqref{eq:stabilityineq} holds  then there exists a spherical suspension $(Z,\sfd_Z,\meas_Z)$ such that:
\begin{equation*}
    \sfd_{mGH}\left((X,\sfd,\meas),(Z,\sfd_Z,\meas_Z)\right)<\eps\, ,
\end{equation*}
where $\sfd_{mGH}$ is the measured-Gromov-Hausdorff distance between two metric measure spaces.
\end{theorem}

\subsubsection*{Related literature}
The roots of the results presented in this note lie in a paper of 1972 by Payne and Rayner\footnote{A curiosity: Dr. Margaret Rayner has been a tutorial fellow in Mathematics at St. Hilda's College-Oxford (the college of the authors of the present note) from the 1960'ies up to her retirement in 1989.}  \cite{PaRa}, establishing the following reverse H\"older inequality for an eigenfunction $u$ of the Dirichlet Laplacian relative to the first eigenvalue $\lambda_{1}(\Omega)$ for a bounded planar domain $\Omega\subset \R^{2}$:
\begin{equation*}
\frac{\|u\|_{L^{2}(\Omega)}} {\|u\|_{L^{1}(\Omega)}} \leq \frac{\sqrt{\lambda_{1}(\Omega)}}{2\sqrt{\pi}}\, .
\end{equation*}
The inequality is sharp and rigid: indeed, equality occurs if and only if $\Omega$ is a disk. The inequality was generalised   to higher dimensional domains in 1973 by Payne and Rayner \cite{PaRa2} and in 1981 by Kohler-Jobin \cite{KoJo}, who proved an isoperimetric-type comparison between the $L^{2}$ and the $L^{1}$ norms of a first eigenfunction of the Dirichlet Laplacian for bounded domains in $\R^{n}$, $n\geq 3$. The inequality was generalised to arbitrary exponents $q\geq p>0$ for bounded domains in $\R^{n}$, $n\geq 2$, by Chiti \cite{Chiti} in 1982. Chiti's  Comparison Theorem was then extended to:
\begin{itemize}
\item domains in the hemisphere, by Ashbaugh-Benguria \cite{AshBen};
\item bounded domains in the hyperbolic space by Benguria-Linde \cite{BeLin};
\item smooth compact Riemannian manifolds with positive Ricci curvature by  Gamara-Hasnaoui-Makni \cite{Gamara} and Colladay-Langford-McDonald  \cite{Colladay};
\item smooth compact Riemannian manifold with an integral Ricci curvature bound by Chen \cite{Chen};
\item smooth Riemannian manifolds with non-negative Ricci curvature and positive asymptotic volume ratio by Chen-Li \cite{ChenLi}.
\end{itemize}
 Let us mention that all the aforementioned papers deal with \emph{smooth} ambient spaces and this note seems to be the first extension of Chiti's comparison result to a \emph{non-smooth} framework.
 \\Moreover, the stable version of Chiti's comparison obtained in Theorem \ref{thm:ChitiStability} seems to be new even in the smooth setting, i.e.  when  the metric measure space $(X,\sfd,\meas)$ is a smooth Riemannian manifold.

\subsubsection*{Acknowledgements}
The authors are  supported by the European Research Council (ERC), under the European's Union Horizon 2020 research and innovation programme, via the ERC Starting Grant  “CURVATURE”, grant agreement No. 802689.
 The authors are grateful to the anonymous reviewer, whose careful reading and comments improved the exposition of the manuscript.

\section{Preliminaries}\label{sec:prel}
\subsection{Cheeger energy and isoperimetric profile}
To state everything mentioned in the previous section in a more precise way, we will need a series of definitions regarding $\mathrm{RCD}(K,N)$ spaces. Throughout this section, we will write $\Lip(X)$ for the set of real valued functions on $X$ that are Lipschitz, and $\Borel(X)$ to denote the set of Borel subsets of $X$. We begin with an assumption that will be made throughout the paper. 
\begin{assumption}
$(X,\sfd,\meas)$ will always denote a compact and separable metric measure space. Indeed, via a generalization of Bonnet-Myers theorem to $\mathrm{CD}(K,N)$ spaces, having $K>0$ guarantees that $(X,\sfd)$ is in fact compact, and further that $\meas(X)<\infty$ so that we may assume without loss of generality, which we henceforth do, that $\meas$ is a probability measure on $X$.
\end{assumption}

We begin by defining the notion of slope on $(X,\sfd,\meas)$:
\begin{definition}
Given a function $u: X\to \R$ we define its slope, denoted by $|\nabla u|$ to be:
\begin{equation}
    |\nabla u|(x_0) =\begin{cases}
    \limsup_{x\to x_0}\frac{|u(x)-u(x_0)|}{\sfd(x,x_0)} & \text{when $x_0$ is not isolated} \\
    0  & \text{when $x_0$ is isolated .}
    \end{cases}
\end{equation}
\end{definition}
\begin{definition}
Let $B\subset X$ be a Borel set and $A\subset X$ be an open set. Then we define the perimeter of $B$ with respect to $A$ as:
\begin{equation}
    \Per(B,A):= \inf_{\{u_n\}_{n\in \setN}} \left\{ \liminf_{n \to \infty} \int_A |\nabla u_n| \di \meas\right\}\, ,
\end{equation}
where the infimum is taken over sequences of functions $\{u_n\}_{n\in \setN}$ such that $u_n \in \Lip(X)$ for all $n \in \mathbb{N}$ and they converge, in $L^1(A,\meas)$, to $\mathcal{X}_A$.\\
For simplicity of notation, we will write $\Per(B):=\Per(B,X)$.
\end{definition}
\begin{definition}
The \textit{isoperimetric profile} corresponding to $(X,\sfd,\meas)$, denoted by $\mathcal{I}_{\meass}:[0,1] \to [0,\infty)$, is given as:
\begin{equation}
    \mathcal{I}_{\meass}(v):=\inf\left\{\Per(E):\meas(E)=v,\, E \in  \Borel(X)\right\}\, .
\end{equation}
We will denote by $\mathcal{I}_{K,N}$, the isoperimetric profile of the model space $\model$ defined in \eqref{eq:defIKN}.
\end{definition}
\begin{definition}
For $u\in L^p(X,\meas)$ and $1<p<\infty$ the $p$-\textit{Cheeger} energy is defined as:
\begin{equation}
    \Ch_p(u):=\inf\left\{\liminf_{n\to \infty} \frac{1}{p}\int_X |\nabla u_n|^p \di \meas:u_n \in \Lip(X)\cap L^p(X,\meas),  \lim_{n \to \infty} \|u_n-u\|_{L^p}=0 \right\}\, .
\end{equation}
We note that the Sobolev space $W^{1,p}\meass$, defined to be the set of functions with finite $p$-\textit{Cheeger} energy, is a Banach space when endowed with the norm:
\begin{equation*}
    \|u\|_{W^{1,p}\meass}:=\left\{\|u\|_{L^p (X,\meas)}^p+p\, \Ch_p(u)\right\}^{1/p}\, .
\end{equation*}
For $\Omega \subset X$, $W^{1,p}_0(\Omega)$ will denote the closure of the set Lipschitz functions with compact support on $\Omega$, with respect to this norm on $W^{1,p}\meass$.\\
We also define, for $u\in W^{1,p}\meass$, $|\nabla u|_w \in L^p(X,\meas)$ to be the function called \textit{minimal weak upper gradient} that gives the following representation for the $p$-\textit{Cheeger} energy (see \cite{AGSRevista}):
\begin{equation*}
    \Ch_p(u)=\frac{1}{p}\int_X |\nabla u|_w^p \, \di \meas\, .
\end{equation*}
\end{definition}
\subsection{Geodesics, $\CD(K,N)$ and $\RCD(K,N)$ conditions}
We let $\Geo(X)$ be the space of geodesics with constant speed. More precisely, we define:
\begin{equation}
    \Geo(X):=\{\gamma :[0,1]\to X \mid \sfd(\gamma(\alpha),\gamma(\beta))=|\alpha-\beta|\sfd(\gamma(0),\gamma(1)), \text{ for all } \alpha,\beta \in [0,1]\}\, .
\end{equation}
For an arbitrary metric space $(Z,\sfd_Z)$, we will write $\Prob(Z)$ to  denote the set of Borel probability measures on $Z$.
For $\mu_1, \mu_2 \in \Prob(Z)$,  let $W_2(\mu_1, \mu_2)$ denote the quadratic transportation distance defined by
\begin{equation}
    W_2(\mu_1,\mu_2):= \inf_{\pi} \left( \int_{X\times X} \sfd^2(x,y)\, \di \pi(x,y) \right)^{1/2}\, ,
    \label{eq:W2distance}
\end{equation}
where the infimum is taken over all Borel probability measures $\pi$ on $Z\times Z$ such that first and second marginal distributions of $\pi$ are $\mu_1$ and $\mu_2$ respectively.
Recall that $(\Prob(Z),W_2)$ is a metric space.
\\

For fixed $t\in[0,1]$, we construct the evaluation map $e_t: \Geo(X)\to X$ as $ e_t(\gamma):=\gamma(t)$.
For $\mu_1,\mu_2 \in \Prob(X)$ we define the set $\Opt(\mu_1,\mu_2) \subset \Prob(\Geo(X))$ to be the set of measures $\nu$ such that $(e_0,e_1)_\sharp \nu$ minimizes the set in \eqref{eq:W2distance}, where we use the subscript $\sharp$ to denote the pushforward of a measure.\\
To be able to state the $\CD(K,N)$ condition we first need to introduce the so-called distortion coefficients. For $0\leq t\leq 1$ and $\theta \in [0,\infty)$, we define:
\begin{equation}
    \tau_{K,N}^{(t)}(\theta)=\begin{cases}
    \infty & \text{when } K\theta^2\geq(N-1)\pi^2\\
    t^{1/N} \left(\frac{\sin\left(t\theta\sqrt{\frac{K}{N-1}}\right)}{\sin\left(\theta\sqrt{\frac{K}{N-1}}\right)}\right)^{\frac{N-1}{N}} & \text{when } 0<K\theta^2<N\pi^2\\
    t^N  & \text{when } K\theta^2<0 \text{ and } N=1 \text{ or when } K\theta^2=0\\
    t^{1/N} \left(\frac{\sinh\left(t\theta\sqrt{\frac{-K}{N-1}}\right)}{\sinh\left(\theta\sqrt{\frac{-K}{N-1}}\right)}\right)^{\frac{N-1}{N}} & \text{when } K\theta^2\leq 0 \text{ and } N>0\, .
    \end{cases}
\end{equation}
We also introduce the \textit{Rényi Entropy functional} $\mathcal{E}:(1,\infty)\times \Prob(X)\to [0,\infty]$ as: 
\begin{equation}
    \mathcal{E}(N,\sigma):= \int_X \rho^{1-\frac{1}{N}} \di \meas\, ,
\end{equation}
where $\sigma=\rho\meas+\sigma^s$, $\sigma^s \perp\meas$.
We are finally in a position to recall the definition of the $\RCD(K,N)$ condition.  Let us briefly mention that the $\CD(K,N)$ condition was introduced independently by Sturm \cite{St1, St2} and Lott-Villani \cite{LottVillani}; the $\RCD$ condition was later introduced by Ambrosio-Gigli-Savar\'e \cite{AGS} in case $N=\infty$ (see also \cite{AGMR}) as a refinement in order to single out ``Riemannian structures'' out of the ``possibly Finslerian $\CD$ structures''.
The finite dimensional counterparts subsequently led to the notions of $\RCD(K,N)$ and $\RCD^*(K, N)$ spaces, as the ``Riemannian refinements'' of  $\CD(K, N)$  and  of the more general $\CD^*(K, N)$ respectively (see \cite{BacherSturm10} for the latter). The class $\RCD(K,N)$ was proposed in \cite{Gigli}. The (a priori more general) $\RCD^*(K,N)$ condition was thoroughly analysed in \cite{EKS} and (subsequently and independently) in \cite{AMS}; see also \cite{CavaMil} for the equivalence between $\RCD^*(K,N)$ and $\RCD(K,N)$ in the case of finite reference measure.

\begin{definition}[$\RCD (K,N)$ condition]
Fix $K\in \R$ and $1<N<\infty$. We say that $\meass$ is a $\CD(K,N)$ space if for all $\mu_0,\mu_1 \in \Prob(X)$ that have bounded support and are aboslutely continuous with respect to $\meas$ we can find $\nu \in \Opt(\mu_0,\mu_1)$ and  an optimal plan $\pi \in \Prob(X \times X)$ such that $\mu_t:=(e_t)_\sharp \nu$ is absolutely continuous with respect to $\meas$ and for all $0\leq t \leq 1$ and $N'\geq N$ it holds:
\begin{equation}
    \mathcal{E}(N',\mu_t)\geq \int \left\{\tau_{K,N'}^{(1-t)}(\sfd(x,y)) \rho_0^{-\frac{1}{N'}}+\tau_{K,N'}^{(t)}(\sfd(x,y)) \rho_1^{-\frac{1}{N'}}\right\} \di \pi(x,y)\, .
\end{equation}
We say that $\meass$ is an $\RCD(K,N)$ space if it is a $\CD(K,N)$ space and further $W^{1,2}\meass$ is a Hilbert space when endowed with an inner product corresponding to $\Ch_2$. In this case, we will write $\Ch(\cdot,\cdot)$ to denote this inner product. 
\end{definition}
Finally, we need to define what it means for a function $u\in W^{1,2}\meass$ to solve the problem \eqref{Laplacian} weakly:
\begin{definition}
    Let $\meass$ be an $\RCD(K,N)$ space for some $K\in (0,\infty),\,  N\in (1,\infty)$, and let $\Omega \subset X$ be an open domain. We say that $u\in W^{1,2}\meass$ is a weak solution of \eqref{Laplacian} if and only if $u\in W_0^{1,2}(\Omega)$ and for all $f\in W_0^{1,2}(\Omega)$ we have that:
    \begin{equation}
        \Ch(u,f)=\int_\Omega \lambda uf \, \di \meas
    \end{equation}
\end{definition}

\subsection{Schwartz symmetrization and the first Dirichlet eigenvalue}

The notion of Schwartz symmetrization recalled below will play a key role in the paper, as it makes an  essential connection between a function on an arbitrary $\mathrm{RCD}(K,N)$ space and a suitable function defined on $(I_{K,N},\sfd_{eu},\meas_{K,N})$, namely its Schwartz symmetrization.

\begin{definition}
 Let $(X,\sfd,\meas)$ be an $\mathrm{RCD}(K,N)$ metric measure space with $K>0$ and $N \in (1,\infty)$, and let $\Omega \subset X$ be an open subset of volume $\meas(\Omega)=v\in (0,1)$. Consider $(I_{K,N},\sfd_{eu},\meas_{K,N})$ defined in \eqref{eq:defIKN}  and let $r(v)>0$ be such that $\meas_{K,N}[0,r(v)]=v$. 
\\For a Borel function $u :\Omega \to \R$,  we define its distribution function $\mu_{u} : \mathbb{R}^+ \to [0,\meas(\Omega)]$ by:
\begin{equation*}
    \mu_u (t) :=\meas\left(\{|u|>t\}\right) \,.
\end{equation*}
  The decreasing rearrangement $u^\#: [0,\meas(\Omega)] \to [0,\infty]$ of $u$ is defined as:
\begin{equation}
    u^\#(s):=\begin{cases}
    \esssup |u| & \text{for } s=0\, , \\
    \inf \{t: \mu_u(t)<s\} & \text{for } s>0\, .
    \end{cases}
\end{equation}
Finally, for $x\in [0,\meas(\Omega)]$, we define the Schwartz symmetrization $u^*$ of $u$ as:
\begin{equation*}
    u^*(x)=u^\#\left(\meas_{K,N}[0,x]\right)\, .
\end{equation*}
 It is immediate to check that $u$ and $u^*$ admit the same distribution function (see for instance \cite[Proposition 3.5]{Polya}); moreover, if $u\in L^p(\Omega,\meas)$ for some $p\in[1,\infty]$ then $u^*\in L^p([0,r(v)]),\meas_{K,N})$ and we have the equality:
\begin{equation}
    \|u\|_{L^p(\Omega,\meas)}=\|u^*\|_{L^p([0,r(v)],\meas_{K,N})}=\|u^\#\|_{L^p([0,v],\mathcal{L}^1)}\, .
\end{equation}
\end{definition}
\begin{remark}
It is trivial to verify that $\mu_u, u^\#$ and $u^*$ are all non-increasing and left continuous.
\end{remark}

With these tools in hand, we will mainly be interested in the solution of the problem \eqref{Laplacian} for an open set $\Omega \subset X$ when $\lambda$ is the smallest positive real number such that a solution exists. We will denote such a $\lambda$ as  $\lambda_X(\Omega)$. Indeed, it is a well-known fact that the first Dirichlet eigenvalue $\lambda_X(\Omega)$ in fact has the variational characterization recalled in Definition \ref{def:VarCharlambda}.

Next, to be able to talk about rigidity and almost rigidity in our results, we need to define what it means for a metric measure space $\meass$ to be a spherical suspension:
\begin{definition}
    Let $(Y,\sfd_Y,\meas_Y)$ and $(Z,\sfd_Z,\meas_Z)$ be geodesic metric measure spaces and let $f:Y\to [0,\infty)$ be Lipschitz. Let $\mathrm{AC}([0,1];Y\times Z)$ denote the set of absolutely continuous curves $\gamma=(\gamma_1,\gamma_2):[0,1]\to Y\times Z$. Let:
    \begin{equation*}
        L_{f}(\gamma):= \int_0^1 \left(|\gamma'_1|^2(t)+\left(f(\gamma_1(t))\right)^2|\gamma'_2|^2(t)\right)^{1/2}\di t\,.
    \end{equation*}
    Define the distance $d$ on $Y\times Z$ as:
    \begin{equation*}
        \sfd_{f}\left((a,b),(c,d)\right):= \inf \left\{L_{f}(\gamma): \gamma\in \mathrm{AC}([0,1];Y\times Z), \gamma(0)=(a,b), \gamma(1)=(c,d)\right\}\, .
    \end{equation*}
    Then, under the equivalence relation $(a,b)\sim(c,d) \iff \sfd_{f}\left((a,b),(c,d)\right)=0$, we obtain a metric space:
    \begin{equation*}
        (Y\times Z/\sim,\sfd_{f}).
    \end{equation*}
    Furthermore, for $N\in[1,\infty)$, endowing this metric space with the measure $\meas_N:= f^N\meas_Y \otimes\meas_Z$, we obtain the metric measure space:
    \begin{equation*}
       Y\times_{f}^N Z:=  (Y\times Z/\sim,\sfd_{f},\meas_N).
    \end{equation*}
   We say that the metric measure space $(X,\sfd,\meas)$  is a spherical suspension over  $\left(Y,\sfd_Y,\meas_Y\right)$  if  $(X,\sfd,\meas)$  is isomorphic to $[0,\pi]\times_{\sin}^{N-1}Y$.
 \end{definition}
  
  It was proved by Ketterer \cite{Kett} that  $[0,\pi]\times_{\sin}^{N-1}Y$ is an  $\RCD(N-1, N)$ space if and only if $\left(Y,\sfd_Y,\meas_Y\right)$ is an $\RCD(N-2, N-1)$ space.

\section{Proofs of the results}

Throughout this section, unless stated otherwise, we will let $(X,\sfd,\meas)$ be an $\mathrm{RCD}(N-1,N)$ space with $\supp(\meas)=X$ and $\meas(X)=1$, $\Omega \subset X$ be an open set such that $\meas(\Omega)=v \in (0,1)$, $\alpha\leq v$ be such that $\lambda_{N-1,N,\alpha}=\lambda_X(\Omega)$, $u$ be a first Dirichlet eigenfunction on $\Omega$ and $z:[0,r(\alpha)]\to \R$ be an eigenfunction corresponding to $\lambda_{N-1,N,\alpha}$. As already recalled in the introduction, such a function $z:[0,r(\alpha)]\to \R$ is uniquely determined, up to scaling by a real constant. 
We begin with the following lemma:
\begin{lemma}\label{lem:zu*}
Let  $z:[0,r(\alpha)]\to \R$ as above be scaled so that
\begin{equation}
    z(0)=u^{*}(0) \, .
\end{equation}
Then:
\begin{itemize}
\item either $\alpha=v$ and $z(s)=u^*(s)$ for every $s\in [0, r(v)]$;
\item  or $\alpha<v$ and  $z(s)\leq u^*(s)$ for every $s\in (0,r(\alpha)]$.
\end{itemize}
\begin{proof}
If $\alpha=v$, we have $\lambda_X(\Omega)=\lambda_{N-1,N,v}$. From the rigidity in \cite[Theorem 1.9]{Polya} we know that $X$ is isomorphic to a spherical suspension and $z(s)=u^*(s)$ for every $s\in [0, r(v)]$. 

Now we turn to the case $\alpha<v$. By the proof of \cite[Theorem 3.10]{Talenti}, one can see that for any $0\leq t'<t\leq \esssup u$ it holds:
\begin{equation*}
    t-t'\leq \int_{\mu(t)}^{\mu(t')}\frac{\lambda_{N-1,N,\alpha}}{\mathcal{I}_{N-1,N}(\xi)^2} \left(\int_0^{\xi} u^{\#}(t) \di t\right) \di \xi\, .
\end{equation*}
Hence, letting $0\leq s<s'\leq \mu(0)$, setting $t=u^{\#}(s)-\epsilon$, $t'=u^{\#}(s')$ for $\epsilon >0$ and then taking $\epsilon \to 0$ (noting that by manipulating equi-measurability, $\mu (u^{\#}(s'))\leq s'$ and $\mu (u^{\#}(s)-\epsilon)\geq s$):
\begin{equation}
 u^{\#}(s)-u^{\#}(s')\leq \int_{s}^{s'}\frac{\lambda_{N-1,N,\alpha}}{\mathcal{I}_{N-1,N}(\xi)^2} \left(\int_0^{\xi} u^{\#}(t) \di t\right) \di \xi.
 \label{eq:bounddiff}
\end{equation}
By \cite[Theorem 3.10]{Talenti} we can see that the right hand-side of \eqref{eq:bounddiff} is in $L^1[0,t)$ for every $t\in(0,v)$ so that we may conclude that $u^\#$ is absolutely continuous on $[0,v)$ and thus almost everywhere differentiable on $(0,v)$. Hence, by \eqref{eq:bounddiff} we have that the following estimate holds for  almost every  $s \in (0,v)$:
\begin{equation}
    -\frac{d}{ds}u^{\#}(s)\leq \frac{\lambda_{N-1,N,\alpha}}{\mathcal{I}_{N-1,N}(s)^2}\int_0^{s} u^{\#}(t) \di t\, .
    \label{uestimate}
\end{equation}
Further, by \cite[Proposition 2.7]{Talenti} we have the following equality,
\begin{equation}
    -\frac{d}{ds}z^{\#}(s)= \frac{\lambda_{N-1,N,\alpha}}{\mathcal{I}_{N-1,N}(s)^2}\int_0^{s} z^{\#}(t) \di t\, , \quad \text{for every $s\in (0,\alpha)$.}
    \label{zestimate}
\end{equation}
 We claim that  $u^{\#}(\alpha)>0$. Indeed,  if by contradiction $u^{\#}(\alpha)=0$, it would follow that $\meas\left(\{u=0\}\right)=v-\alpha>0$ which gives a contradiction since $u$ is an eigenfunction on $\Omega$ and thus needs to be strictly positive (see \cite[Corollary 5.7]{Latvala}). Now, from the above estimates we know that $u^\#$ and $z^\#$ are continuous functions on $[0,\alpha]$, with $u^\#>0$ on   $[0,\alpha]$. Thus, $\frac{z^{\#}(s)}{u^{\#}(s)}$ is also a continuous function on $[0,\alpha]$ and we can set:
\begin{equation*}
    a:= \max_{s\in [0,\alpha]} \frac{z^{\#}(s)}{u^{\#}(s)}.
\end{equation*}
Assume by contradiction that $a>1$, and set:
\begin{equation}
    s_0:= \inf\{ s\in (0, \alpha] : a u^\#(s)=z^\#(s)\}.
\end{equation}
Clearly it holds that $s_0>0$, since $u^\#(0)=z^\#(0)$ and both functions are right-continuous at $0$. Next, define:
\begin{equation}
    w^{\#}(s):=
  \begin{cases}
     a u^{\#}(s), & \text{for } 0 \leq s \leq s_0 \\
    z^{\#}(s), & \text{for } s_0 \leq s \leq \alpha \,.
  \end{cases}
\end{equation}
Letting $w(x):=w^{\#}(\meas_{N-1,N}[0,x])$, we can see that:
\begin{align*}
    \int_0^{r(\alpha)} \left(w'\right)^2 \di \meas_{N-1,N}=\int_0^{r(\alpha)} \sin^{2N-2}(x)\left(w^{\# '}\left(\meas_{N-1,N}[0,x]\right)\right)^2 \di \meas_{N-1,N}\\=
    \int_0^{r(\alpha)} \left(w^{\# '}\left(\meas_{N-1,N}[0,x]\right)\right)^2 \mathcal{I}_{N-1,N}(\meas_{N-1,N}[0,x])^2 \di \meas_{N-1,N}\, .
\end{align*}
By \eqref{uestimate} and \eqref{zestimate}, we have:
\begin{equation}
     -\frac{d}{ds}w^{\#}(s)\leq \frac{\lambda_{N-1,N,\alpha}}{\mathcal{I}_{N-1,N}(s)^2}\int_0^{s} w^{\#}(t)\, \di t, \quad \text{for a.e.  $s\in (0,s_0)\cup(s_0,\alpha)$ ,}
    \label{eq:vestimate}
\end{equation}
and hence:
\begin{align*}
     \int_0^{r(\alpha)} \left(w'\right)^2 \di \meas_{N-1,N} &\leq \lambda_{N-1,N,\alpha}\int_0^{r(\alpha)}(-w^{\# '}(\meas_{N-1,N}[0,x]))\left(\int_0^{\meas_{N-1,N}[0,x]}w^{\#}(\zeta)d\zeta\right)  \di \meas_{N-1,N} \\
&=\lambda_{N-1,N,\alpha} \int_0^{r(\alpha)}\left(w^{\#}\left(\meas_{N-1,N}[0,x]\right)\right)^2 \di \meas_{N-1,N} \nonumber \\
&= \lambda_{N-1,N,\alpha} \int_0^{r(\alpha)} w^2 \di \meas_{N-1,N}\, , \nonumber
\end{align*}
where we have used the estimate \eqref{eq:vestimate} and the fact $w^{\# '}\leq 0$ for the inequality, and integrated by parts for the first equality noting that we have $w^{\#}(\alpha)=0$.
Thus, we get the following inequality:
\begin{equation}
    \lambda_{N-1,N,\alpha}\geq \frac{ \int_0^{r(\alpha)} (w')^2 \di \meas_{N-1,N}}{\int_0^{r(\alpha)} w^2 \di \meas_{N-1,N} }\, .
\end{equation}
It follows that $w$  minimizes the Rayleigh quotient and therefore it is an eigenfunction associated to $\lambda_{N-1,N,\alpha}$ on $[0,r(\alpha)]$. Henceforth, $w$ is a multiple of $z$ and thus $ a u^{\#}(s)= z^{\#}(s)$ on $s\in [0,s_0]$. Since $u^{\#}(0)=u^*(0)=z(0)=z^{\#}(0)$, we conclude that $a=1$, contradicting that $a>1$.
\label{chitilemma}
\end{proof}
\end{lemma}

We can now state and prove the generalization of Chiti's comparison Theorem to $\mathrm{RCD}(K,N)$ spaces.

\begin{theorem}[Chiti's Comparison Result in $\mathrm{RCD}(K,N)$ Spaces]
Fix $p>0$. Let $(X,\sfd,\meas)$ be an $\mathrm{RCD}(N-1,N)$ space, $\Omega \subset X$ be an open set such that $\meas(\Omega)=v \in (0,1)$, $\alpha\leq v$ be such that $\lambda_{N-1,N,\alpha}=\lambda_X(\Omega)$, $u$ be a first Dirichlet eigenfunction on $\Omega$ and $z:[0,r(\alpha)]\to \R$ be an eigenfunction corresponding to $\lambda_{N-1,N,\alpha}$ scaled so that:
\begin{equation}
    \int_{\Omega}u^p \di \meas = \int_0^{r(v)} u^{* p} \di \meas_{N-1,N}=\int_0^{r(\alpha)} z^p \di \meas_{N-1,N}.
    \label{eq:chitiintegrals}
\end{equation}
Then there exists $r_1 \in (0,r(\alpha))$ such that:
$$
\begin{cases}
      u^{*}(s) \leq z(s) & \text{for } 0 \leq s \leq r_1, \\
    z(s) \leq u^{*}(s) & \text{for } r_1 \leq s \leq r(\alpha).
  \end{cases}
  $$
\begin{proof}
The case $\alpha=v$ is trivial since it implies, as mentioned above, that $z(s)=u^*(s)$ for every $s\in [0, r(v)]$. Thus suppose $\alpha<v$. If $z(0)<u^{*}(0)$, then there exists $a>1$ such that $az(0)=u^{*}(0)$. But then in that case, by Lemma \ref{lem:zu*}, it holds $az(s)\leq u^{*}(s)$ for all $s\in [0,r(\alpha)]$ in which case the equality of integrals in \eqref{eq:chitiintegrals}  would contradict that $a>1$. Hence, we must have $z(0)\geq u^{*}(0)$.\\
Assume $z(0)=u^{*}(0)$. Then by Lemma \ref{lem:zu*}, we see that the second equality in \eqref{eq:chitiintegrals} can hold only if $\alpha=v$, in which case the statement trivially holds as we discussed above.\\

Now let us focus on the case $z(0)>u^{*}(0)$, where we know $\alpha<v$ holds. Firstly, note that if $z(s)$ minimizes the Rayleigh quotient, so does $z^{*}(s)$, thus  we must have $z=z^*$. Recall, by the estimate in \eqref{uestimate} and the corresponding estimate for $z^\#$, that  $u^\#$ and $z^\#$ are both absolutely continuous on $[0,\alpha)$. Next, define the set:
 \begin{equation*}
     S_1 :=  \{ s\in [0, \alpha] : u^\#(s)=z^\#(s), u^{\#}(t) \leq z^{\#}(t) \text{ for all } t\in [0,s]\}.
 \end{equation*}
and let $s_1$ be the supremum of this set.
Since $u^\#$ and $z^\#$ are both continuous and  $z^{\#}(\alpha)=0$ and $u^{\#}(\alpha)>0$ (see the proof of Lemma \ref{chitilemma}),  clearly this set is non-empty and closed  so that it includes its maximal element $s_1 \in (0,\alpha)$.

We further define:
\begin{equation}
   S_2 := \{ s\in (s_1, \alpha) : u^\#(s)=z^\#(s) \}\, . 
   \label{eq:s2def}
\end{equation}
Suppose this set is non-empty and let $s_2$ be the infimum of it.  By the continuity of the rearrangements, we infer that the set $S_2$ is closed and thus it contains its minimal element $s_2$. 

We claim that $s_2>s_1$.  Indeed, by continuity of the rearrangements, for $\epsilon>0$ small enough, we have: 
\begin{equation*}
    \int_0^s u^\#(t) - z^\# (t) \di t <0\, , \quad \text{ for all $s\in (0,s_1+\epsilon]$}\,.
\end{equation*}
Hence, using \eqref{uestimate} and \eqref{zestimate}, we obtain that for all  $s\in (s_1,s_1+\epsilon]$ it holds:
\begin{equation*}
  z^\#(s)-u^\#(s) = \int_{s_1}^s \frac{d}{\di t} \left( z^\#(t)-u^\#(t)\right) \di t  \leq  \int_{s_1}^s \frac{\lambda_{N-1,N,\alpha}}{I_{N-1,N}(t)^2} \left( \int_0^t u^\#(\xi)-z^\#(\xi) \di \xi \right) \di t<0\, .
\end{equation*}
Thus $z^\#(s)<u^\#(s)$ on $(s_1,s_1+\epsilon]$ and hence, by definition, $s_2>s_1$.

We next claim it must hold
$$u^{\#}(t)>z^\#(t), \quad \text{ for all $t \in (s_1,s_2)$ . }$$
Indeed,  otherwise we would have $u^{\#}(t)<z^\#(t)$ for all $t \in (s_1,s_2)$, which would contradict the definition of $s_1$ since $s_2>s_1$. Set:
\begin{equation*}
    w^{\#}(s):=
  \begin{cases}
     z^{\#}(s), & \text{for } s \in [0,s_1]\cup [s_2,\alpha]  \\
    u^{\#}(s), & \text{for } s\in[s_1,s_2]\, .
  \end{cases}
\end{equation*}
Again, by \eqref{uestimate}, \eqref{zestimate}, we see that for a.e. $s\in (0,s_1)\cup(s_1,s_2)\cup(s_2,\alpha)$ it holds:
\begin{equation}
     -\frac{d}{ds}w^{\#}(s)\leq \frac{\lambda_{N-1,N,\alpha}}{\mathcal{I}_{N-1,N}(s)^2}\int_0^{s} w^{\#}(t) \di t\, .
    \label{eq:westimate}
\end{equation}
Thus, setting $w(x):=w^{\#}(\meas_{N-1,N}[0,x])$ and proceeding as in the proof of Lemma \ref{chitilemma}, it follows that $w:[0, r(\alpha)]\to \R$ also minimizes the Rayleigh quotient corresponding to $\lambda_{N-1,N,\alpha}$,  thus $w^\#=z^\#$  which in turn implies that $u^\#(s)=z^\#(s)$ for all $s\in[s_1,s_2]$, which again contradicts the definition of $s_1$.
\\ Therefore the set $S_2$ in \eqref{eq:s2def} must be empty, i.e. it must be the case that $u^\#(s)>z^\#(s)$ on $s\in (s_1,\alpha]$.
\end{proof}
\label{ChitiThm}
\end{theorem}
To obtain the reverse Hölder inequality as a corollary of the theorem above, one can use the following Lemma:
\begin{lemma}
Let $b>0$, and let $\phi:\mathbb{R}\to \mathbb{R}$ be a convex function. Let $f,g :\mathbb{R}\to \mathbb{R}$ be functions that satisfy the following conditions for all $y \in \mathbb{R}$:
\begin{align}
   \int_0^b g(x) \, \di x&=  \int_0^b f(x)\, \di x\\
   \int_0^b \{g(x)-y\}^+ \di x &\leq  \int_0^b \{f(x)-y\}^+ \di x 
\end{align}
where $\{h(x)\}^+:= \max(h(x),0)$, then:
\begin{equation}
   \int_0^b \phi\left(g(x)\right) \di x \leq  \int_0^b \phi\left(f(x)\right) \di x.  
\end{equation}
\begin{proof}
This is a special case of \cite[Theorem 9]{Hardy}.
\end{proof}
\label{hardylemma}
\end{lemma}
Now, we are finally in a position to prove our first main result.
\begin{proof}[Proof of Theorem \ref{mainthm}]
As in the previous arguments, the case $\alpha=v$ means $\meass$ is a spherical suspension and hence the result is trivial. Now let us focus on the case $\alpha<v$. Extend $z^\#(s)$ to $[0,v]$ by taking it to be $0$ on $(\alpha,v)$. By Theorem \ref{ChitiThm}, we know that there exists $s_1$ such that $u^\#(s)\leq z^\#(s)$ for $s\in [0,s_1]$ and $z^\#(s)\leq u^\#(s)$ for $s\in [s_1,v]$. 
\\Now, if $s\in [0,s_1]$:
\begin{equation*}
    \int_0^s \left(u^\#(t)\right)^p \di t \leq \int_0^s \left(z^\#(t)\right)^p \di t.
\end{equation*}
If instead $s\in [s_1,v]$:
\begin{align*}
    \int_0^s \left(u^\#(t)\right)^p \di t &=\int_0^v \left(u^\#(t)\right)^p \di t-\int_s^v \left(u^\#(t)\right)^p \di t  \leq \int_0^v \left(u^\#(t)\right)^p \di t-\int_s^v \left(z^\#(t)\right)^p \di t \\
    &=\int_0^v \left(z^\#(t)\right)^p \di t-\int_s^v \left(z^\#(t)\right)^p \di t = \int_0^s \left(z^\#(t)\right)^p \di t\, .
\end{align*}
Hence, 
\begin{equation}\label{eq:intuleqintz}
    \int_0^s \left(u^\#(t)\right)^p \di t \leq \int_0^s \left(z^\#(t)\right)^p \di t, \quad \text{for all } s\in[0,v]\, .
\end{equation}
From the fact that $u^\#$ is non-increasing, it follows that for every $y \in \mathbb{R}$ there exists $\xi \in[0,v]$ such that:
\begin{align}
    \int_0^v \{\left(u^\#(t)\right)^p -y\}^+ \di t =\int_0^{\xi} \left(u^\#(t)\right)^p -y \, \di t \leq\int_0^{\xi} \left(z^\#(t)\right)^p -y\,  \di t \\
    \leq \int_0^{\xi} \{\left(z^\#(t)\right)^p -y\}^+ \di t \leq \int_0^v \{\left(z^\#(t)\right)^p -y\}^+ \di t.
\end{align}
Thus the conditions of Lemma \ref{hardylemma} are satisfied.  Taking $\phi(x)=x^{\frac{q}{p}}$ for $q\geq p$, we can conclude that:
\begin{equation}
    \int_0^v \left(u^\#(t)\right)^q \di t \leq \int_0^v \left(z^\#(t)\right)^q \di t,
\end{equation}
from which the inequality \eqref{eq:holderineq} follows.\\
To get the rigidity statement, assume that equality holds for some fixed $q>p$. We claim that then equality holds for all $p'\in(p,q)$. To see this, note that if
\begin{equation}
    \int_0^v \left(u^\#(t)\right)^q \di t=\int_0^\alpha \left(z^\#(t)\right)^q \di t\, ,
\end{equation}
then applying Lemma \ref{hardylemma} with $\phi(x)=-x^{\frac{p'}{q}}$ (conditions of the lemma are satisfied again as shown above) for any $p'\in (p,q)$ we get that:
\begin{equation*}
    \int_0^v \left(z^\#(t)\right)^{p'} \di t\leq \int_0^v \left(u^\#(t)\right)^{p'} \di t\, .
\end{equation*}
Combining this with inequality \eqref{eq:holderineq} we get, for all $p'\in (p,q)$ that:
\begin{equation}
 \|u\|_{L^{p'}\left(\Omega,\meas\right)}= \|z\|_{L^{p'}\left([0,\alpha],\meas_{N-1,N}\right)} \, .
\end{equation}
Thus, by virtue of \cite[Theorem 1]{Klun}, it must be the case that $u$ and $z$ have the same distribution functions so that, in fact, $u^\#(s)=z^\#(s)$. But if $\alpha<v$ we have that $u^\#(\alpha)>0$ and $z^\#(\alpha)=0$. Thus, we must have $\alpha=v$, in which case $X$ is isomoprhic to a spherical suspension. The proof of Theorem \ref{mainthm} is now complete.
\end{proof}
In the proof the Stability  Theorem \ref{thm:ChitiStability}, we will make use of the co-area formula stated as in the next proposition (see for instance \cite[Proposition 3.7]{Talenti} and  \cite[Remark 4.3]{Miranda} for more details).

\begin{proposition}[Coarea formula]
Let $K\in \mathbb{R}$ and $N \in (1,\infty)$, let $(X,\sfd,\meas)$ be an $\mathrm{RCD}(K,N)$ space and let $\Omega \subset X$ be an open domain. If $u : \Omega \to [0,\infty)$ is an element of $W_0^{1,2}(\Omega)$, then for any $t > 0$ we have:
\begin{equation}
    \int_{\{u>t\}} \left|\nabla u\right|_{w} \di \meas= \int_t^\infty \Per\left(\{u>s\}\right) \di s\, .
    \label{eq:coarea}
\end{equation}
\end{proposition}
We are now in position to prove Theorem \ref{thm:ChitiStability}.
\begin{proof}[Proof of Theorem \ref{thm:ChitiStability}]
Firstly note that if $Q$ is an unbounded subset of $(1,\infty)$, then having
\begin{equation*}
    \|z\|_{L^q\left((0,r(\alpha)),\meas_{N-1,N}\right)}-\|u^*\|_{L^q\left(0,r(v)),\meas_{N-1,N}\right)}<\delta
\end{equation*}
for all $q\in Q$ implies:
\begin{equation}\label{eq:z0-u0}
    z^\#(0)-u^\#(0)=\|z\|_{L^\infty(0,r(\alpha))}-\|u^*\|_{L^\infty(0,r(v))}<\delta.
\end{equation}
Now by the assumptions of the theorem, there exists $s_1 \in (0,v)$ such that:
\begin{equation}\label{eq:int0s1=ints1v}
    \int_0^{s_1}z^\#(s)- u^\#(s) \di s = \int_{s_1}^{v}u^\#(s)- z^\#(s) \di s.
\end{equation}
Furthermore, using the estimates \eqref{uestimate} and \eqref{zestimate} and applying \eqref{eq:intuleqintz} with $p=1$ one can infer:
\begin{equation*}
    \frac{d}{ds}\left(z^\#(s)-u^\#(s)\right) \leq \frac{\lambda_{N-1,N,\alpha}}{\mathcal{I}_{N-1,N}(s)^2}\int_0^{s} u^{\#}(t) - z^\#(t) \di t \leq 0, \quad \text{for a.e. $s \in [0,v]$.}
\end{equation*}
Thus  $z^\#(s)-u^\#(s)$ is decreasing and,  using \eqref{eq:z0-u0} -\eqref{eq:int0s1=ints1v},  we get the estimate:
\begin{multline*}
    \int_\alpha^{\alpha+\sqrt{\delta}} u^\#(s) \di s =\int_\alpha^{\alpha+\sqrt{\delta}} u^\#(s) -z^\#(s) \di s \leq \int_{s_1}^v u^\#(s)- z^\#(s) \di s \\ =\int_0^{s_1}z^\#(s)- u^\#(s) \di s \leq s_1 (z^\#(0)-u^\#(0)) \leq v \delta.
\end{multline*}
Hence, since $u^\#$ is continuous, by the mean value theorem for integrals, there exists $y \in (\alpha,\alpha+\sqrt{\delta})$ such that:
\begin{equation} \label{eq:yestimate}
  u^\#(y)= \frac{1}{\sqrt{\delta}}\int_\alpha^{\alpha+\sqrt{\delta}} u^\#(s) \di s \leq v\sqrt{\delta} \, .
\end{equation}
Differentiating both sides of \eqref{eq:coarea} one can see that:
\begin{equation}
    -\frac{d}{dt}\left( \int_{\{u>t\}} \left|\nabla u\right|_{w} \di \meas\right)= \Per\left(\{u>t\}\right)\, , \quad \text{for a.e. $t$}\, .
\end{equation}
In particular, display (41) in \cite{Talenti} reads (from here on, for simplicity of notation, we write $\lambda$ to denote $\lambda_X(\Omega)=\lambda_{N-1,N,\alpha}$):
\begin{equation*}
    \mathcal{I}_{N-1,N}(\mu_u(t))^2\leq \left(\Per\left(\{u>t\}\right)\right)^2 \leq -\lambda \mu_u'(t) \int_0^{\mu_u(t)} u^{\#}(s) \di s\, .
\end{equation*}
Dividing the inequality by $ \mathcal{I}_{N-1,N}(\mu_u(t))^2$, we obtain:
\begin{equation*}
    \frac{\left(\Per\left(\{u>t\}\right)\right)^2}{ \mathcal{I}_{N-1,N}(\mu_u(t))^2}\leq \frac{-\lambda \mu_u^{'}(t)}{\mathcal{I}_{N-1,N}(\mu_u(t))^2} \int_0^{\mu_u(t)} u^{\#}(s) \di s\, .
\end{equation*}
Integrating  from $u^\#(y)$ to $u^\#(0)$ yields:
\begin{equation*}
   \int_{u^\#(y)}^{u^\#(0)} \frac{\left(\Per\left(\{u>t\}\right)\right)^2}{ \mathcal{I}_{N-1,N}(\mu_u(t))^2} \di t \leq \int_{u^\#(y)}^{u^\#(0)}\frac{-\lambda \mu_u^{'}(t)}{\mathcal{I}_{N-1,N}(\mu_u(t))^2} \int_0^{\mu_u(t)} u^{\#}(s) \di s \di t\, .
\end{equation*}
Performing the change of variable  $\xi=\mu_u(t)$, noting that $\mu_u(u^\#(y))\leq y\leq \alpha +\sqrt{\delta}$ and applying \eqref{eq:intuleqintz} it follows that:
\begin{equation}\label{eq:3.30}
    \int_{u^\#(y)}^{u^\#(0)} \frac{\left(\Per\left(\{u>t\}\right)\right)^2}{ \mathcal{I}_{N-1,N}(\mu_u(t))^2} \di t \leq \int_0^{\alpha+\sqrt{\delta}}\frac{\lambda}{\mathcal{I}_{N-1,N}(\xi)^2}\int_0^\xi z^\#(s) \di s \di \xi\, .
\end{equation}
We next estimate the right hand side of \eqref{eq:3.30} by splitting the integral in the two contributions: on $[0, \alpha]$ and on $[\alpha, \alpha+\delta]$.
The former contribution is controlled by using \eqref{zestimate}:
\begin{equation}\label{eq:3.31}
    \int_0^{\alpha}\frac{\lambda}{\mathcal{I}_{N-1,N}(\xi)^2}\int_0^\xi z^\#(s) \di s \di \xi= z^\#(0)-z^\#(\alpha)=z^\#(0). 
\end{equation}
 Noting that $z^\#=0$ on $(\alpha,\alpha+\sqrt{\delta})$ and that anything that depends on $\alpha$ is in fact fixed by fixing $\lambda$, the latter contribution can be estimate as:
\begin{align}
    \int_\alpha^{\alpha+\sqrt{\delta}}\frac{\lambda}{\mathcal{I}_{N-1,N}(\xi)^2}\int_0^\xi z^\#(s)\di s \di \xi &= \int_\alpha^{\alpha+\sqrt{\delta}}\frac{\lambda}{\mathcal{I}_{N-1,N}(\xi)^2}\int_0^\alpha z^\#(s) \di s \di \xi  \nonumber \\ 
    &\leq C(N,v,\lambda)\sqrt{\delta} \, , \label{eq:3.32}
\end{align}
where in the last equality we used that $\mathcal{I}_{N-1,N}$ does not vanish on the compact set $[\alpha,v]$ and hence is bounded below by a constant that only depends on $N,\lambda,v$.
 Combining \eqref{eq:3.30}, \eqref{eq:3.31} and \eqref{eq:3.32} gives that (from here on $C(N,v,\lambda)$ may change from line to line):
\begin{equation*}
    \frac{1}{u^\#(0)-u^\#(y)}\int_{u^\#(y)}^{u^\#(0)} \frac{\left(\Per\left(\{u>t\}\right)\right)^2}{ \mathcal{I}_{N-1,N}(\mu_u(t))^2} \di t \leq \frac{z^\#(0)+ C(N,v,\lambda) \sqrt{\delta}}{u^\#(0)-u^\#(y)}\, ,
\end{equation*}
so that there exists $t_0 \in (u^\#(y),u^\#(0))$ such that:
\begin{align}
    \frac{\left(\Per\left(\{u>t_0\}\right)\right)^2}{ \mathcal{I}_{N-1,N}(\mu_u(t_0))^2} & \leq \frac{z^\#(0)+ C(N,v,\lambda) \sqrt{\delta}}{u^\#(0)-u^\#(y)} \nonumber \\ 
    &\leq \frac{u^\#(0)+ \delta+ C(N,v,\lambda) \sqrt{\delta}}{u^\#(0)-u^\#(y)} = 1 +\frac{u^\#(y)+\delta+ C(N,v,\lambda) \sqrt{\delta}}{u^\#(0)-u^\#(y)} \, .  \label{eq:3.33}
\end{align}
Now, pick $\tilde{\delta}>0$ in the statement  of Theorem \ref{thm:ChitiStability} such that
\begin{equation*}
    z^\#(0)-\tilde{\delta}>\frac{z^\#(0)}{2} \quad \text{and} \quad 
    \frac{z^\#(0)}{2}-v\sqrt{\tilde{\delta}}>\frac{z^\#(0)}{4}\, .
\end{equation*}
Note  that $\tilde{\delta}>0$  depends only on $N,v,\lambda$. Then,  using \eqref{eq:z0-u0} and \eqref{eq:yestimate},  we get:
 \begin{equation*}
    u^\#(0)-u^\#(y)>z^\#(0)- \delta -u^\#(y)\geq \frac{z^\#(0)}{2}-v\sqrt{\delta}\geq\frac{z^\#(0)}{4}\, ,  \quad \text{for any  $\delta \in (0,\tilde{\delta})$}\, .
 \end{equation*}
 Plugging \eqref{eq:yestimate} and the last estimate into \eqref{eq:3.33} yields:
\begin{equation*}
     \frac{\left(\Per\left(\{u>t_0\}\right)\right)^2}{ \mathcal{I}_{N-1,N}(\mu_u(t_0))^2} \leq 1+ C(N,v,\lambda) (\sqrt{\delta} +\delta)\,, \quad \text{for any $\delta \in (0,\tilde{\delta})$}\, .
\end{equation*}
Picking $\tilde{\delta}>0$ smaller if necessary, but again only depending on $N,v,\lambda$,  we infer that 
\begin{equation}\label{eq:finalStab}
     \frac{\left(\Per\left(\{u>t_0\}\right)\right)^2}{ \mathcal{I}_{N-1,N}(\mu_u(t_0))^2} \leq 1+ C(N,v,\lambda) \sqrt{\delta} \,, \quad \text{for any $\delta \in (0,\tilde{\delta})$}\, .
\end{equation}
Combining \eqref{eq:finalStab} with the B\'erard-Besson-Gallot-type quantitative improvement of  the L\'evy-Gromov isoperimetric inequality obtained for the present non-smooth framework in  \cite{Obata}[Lemma 3.1, 3.2], gives that there exists some constant $B=B(N)>0$ and $\tilde{\delta}= \tilde{\delta}(N,v,\lambda)>0$ such that
\begin{equation}
    B \left(\pi-\text{diam}(X)\right)^N \leq C(N,v,\lambda)\sqrt{\delta} \,, \quad \text{for any $\delta \in (0,\tilde{\delta})$}\, .
\end{equation}

This proves the first part of the statement. The second part follows from the first one by a standard compactness/contradiction argument, which uses the compactness of the class of $\mathrm{RCD}(N-1,N)$ spaces with respect to the measured-Gromov-Hausdorff topology and the maximal diameter Theorem \cite[Theorem 1.4]{Kett}.
\end{proof}

{
\footnotesize

{}

}
\end{document}